\documentclass[11pt,leqno]{amsart}
\usepackage{amssymb,multirow}

\newcommand{\pde}[2]{\ensuremath{\left\{\begin{array}{l}#1\vspace{0.2cm}\\#2\end{array}\right.}}
\newcommand{\R}{\ensuremath{\mathbb{R}}}
\newcommand{\C}{\ensuremath{\mathbb{C}}}
\newcommand{\re}{\ensuremath{\mathrm{Re}}}
\newcommand{\supp}{\ensuremath{\mathrm{supp}\,}}
\newcommand{\so}{\ensuremath{^{\{1\}}}}
\newcommand{\st}{\ensuremath{^{\{2\}}}}
\newcommand{\sk}{\ensuremath{^{\{i\}}}}
\newcommand{\sn}{\ensuremath{^{(n)}}}
\newcommand{\snp}{\ensuremath{^{(n+1)}}}
\newcommand{\snm}{\ensuremath{^{(n-1)}}}
\newcommand{\snmm}{\ensuremath{^{(n-2)}}}
\newcommand{\sm}{\ensuremath{^{(-1)}}}
\newcommand{\xk}{\ensuremath{\mathcal{X}}}

\newtheorem{thrm}{Theorem}[section]
\newtheorem{lem}[thrm]{Lemma}
\newtheorem{propn}[thrm]{Proposition}
\newtheorem{cor}[thrm]{Corollary}

\theoremstyle{remark}

\newtheorem{rk}[thrm]{Remark}

\numberwithin{equation}{section}

\title[Large data local well-posedness II]{Large data local well-posedness for a class of KdV-type equations II}
\author[B. Harrop-Griffiths]{Benjamin Harrop-Griffiths}
\address{Department of Mathematics, University of California, Berkeley, CA 94720}
\email{benhg@math.berkeley.edu}


\begin{document}

\maketitle


\begin{abstract}
We consider the Cauchy problem for an equation of the form
\[
(\partial_t+\partial_x^3)u=F(u,u_x,u_{xx})
\]
where \(F\) is a polynomial with no constant or linear terms and no quadratic \(uu_{xx}\) term. For a polynomial nonlinearity with no quadratic terms, Kenig-Ponce-Vega proved local well-posedness in \(H^s\) for large \(s\). In this paper we prove local well-posedness in low regularity Sobolev spaces and extend the result to certain quadratic nonlinearities. The result is based on spaces and estimates similar to those used by Marzuola-Metcalfe-Tataru for quasilinear Schr\"odinger equations.
\end{abstract}


\section{Introduction}
We consider local well-posedness for the Cauchy problem
\begin{equation}\label{eq:pde}
\pde{(\partial_t+\partial_x^3)u=F(u,u_x,u_{xx})\qquad u\colon\R\times\R\rightarrow\R\textrm{ or }\C}{u(0)=u_0}
\end{equation}
where we assume \(F\) is a constant coefficient polynomial of degree \(m\geq 2\) with no constant or linear terms.

It is natural to consider well-posedness in \(H^s(\R)\). However, due to the infinite speed of propagation, even a linear equation
\[
(\partial_t+\partial_x^3+a(x)\partial_x^2)u=0
\]
where \(a\) is smooth with bounded derivatives requires a Mizohata-type necessary condition for \(L^2\) well-posedness \cite{A,Mz,Tam}
\begin{equation}\label{eq:Mz}
\sup_{x_1\leq x_2}\re\,\int_{x_1}^{x_2} a(x)\,dx<\infty
\end{equation}
So at the very least, when \(F\) contains a term of the form \(uu_{xx}\) we expect any solution \(u\) to \eqref{eq:pde} to require some additional integrability. Indeed, an ill-posedness result in \(H^s\) was proved by Pilod \cite{P}. Local well-posedness was established using weighted spaces \(H^s\cap L^2(|x|^kdx)\) for sufficiently large \(k\in\mathbb{Z}_+\) and \(s>0\) by Kenig-Ponce-Vega \cite{KPV1,KPV2} and in the case of systems by Kenig-Staffilani \cite{KS}. Several authors have considered quasilinear versions of the problem for which \eqref{eq:pde} is a special case (see \cite{A}, \cite{C} and references therein).

By replacing weighted spaces with a spatial summability condition, \linebreak Marzuola-Metcalfe-Tataru \cite{MMT} proved a small data result for quasilinear Schr\"odinger equations in a translation invariant subspace \(l^1H^s\subset H^s\) using a similar space to one suggested in \cite{KPV3}. In \cite{HG} the author adapted this result to \eqref{eq:pde} and using a similar method to Bejenaru-Tataru \cite{BT} was able to prove the result for large initial data.

As in \cite{KPV1, KPV2, KPV3, KS, MMT, MMT2} we expect to be able to consider initial data in \(H^s\) when \(F\) contains no quadratic terms. In fact, the Mizohata condition \eqref{eq:Mz} and ill-posedness results of Pilod \cite{P} suggest that we should be able to establish well-posedness provided \(F\) contains no quadratic \(uu_{xx}\) term. Our main result is that this is indeed the case.

\vspace{15pt}\begin{thrm}\label{thrm:main}Suppose \(F\) contains no terms of the form \(uu_{xx}\). Then for all \(s>\tfrac{9}{2}\) there exists some \(C>0\) such that the equation \eqref{eq:pde} is locally well-posed in \(H^s\) on the time interval \([0,T]\) where \(T=e^{-C\|u_0\|_{H^s}}\).\end{thrm}\vspace{15pt}

We take the definition of ``well-posedness" to be the existence and uniqueness of a solution \(u\in C([0,T],H^s(\R))\) to \eqref{eq:pde} and Lipschitz continuity of the solution map
\[
H^s\ni u_0\mapsto u\in C([0,T],H^s(\R))
\]

\vspace{15pt}\begin{rk}In the case that \(u\) is complex valued, we consider ``terms of the form \(uu_{xx}\)" to include the terms \(\overline uu_{xx},u\overline u_{xx},\overline u\overline u_{xx}\). In the proof of Theorem \ref{thrm:main} we will ignore complex conjugates, but it will be clear from the proof that \(F\) can be a polynomial in \(u,\overline u,u_x,\overline u_x,u_{xx},\overline u_{xx}\).\end{rk}\vspace{15pt}

The proof of Theorem \ref{thrm:main} is similar to \cite{HG} with key new ingredients being trilinear estimates similar to those proved in \cite{MMT2} and a linear estimate for a system of frequency localised equations. As our function spaces are adapted to the unit time interval, following Bejenaru-Tataru \cite{BT}, we split the data into low and high frequency components. The low frequency component of the initial data \(u_0^{l}\) is essentially stationary on the unit interval so we fix it at \(t=0\) and solve an equation for the high frequency part of the solution \(v=u-u_0^{l}\),
\[
\pde{(\partial_t+\partial_x^3)v=\tilde F(x,v,v_x,v_{xx})}{v(0)=u_0^{h}=u_0-u_0^l}
\]
By rescaling the initial data we can ensure the high frequency component of the initial data \(u_0^{h}\) is small and hence we can solve for \(v\) using a perturbative argument. The Mizohata condition \eqref{eq:Mz} suggests that the quadratic terms involving \(v_{xx}\) behave in a quasilinear manner. In order to handle this, we use a paradifferential decomposition similar to Marzuola-Metcalfe-Tataru \cite{MMT,MMT2} to break the equation into a system of frequency localised equations of the form
\begin{equation}\label{eq:introsystem}
(\partial_t+\partial_x^3-\partial_xa_{<j}\partial_x^2)v_j=f_j
\end{equation}
As in \cite{HG}, we solve this by conjugating the principal part by a suitable exponential term and find an approximate solution by solving a linear Airy equation.

The remainder of the paper is structured as follows. In Section \ref{sect:spaces} we define the function spaces used to prove Theorem \ref{thrm:main} and prove a number of estimates. In Section \ref{sect:lin} we prove estimates for the system of frequency localised equations \eqref{eq:introsystem}. In Section \ref{sect:rescale} we discuss the splitting of the initial data and rescaling. In Section \ref{sect:pf} we complete the proof of Theorem \ref{thrm:main}.

\vspace{15pt}
\begin{rk}While this result covers the case of the KdV, mKdV and gKdV, it is far from the best known results for these equations and we refer the reader to \cite{LinPon} for a summary of results and references.

However, as in \cite{HG}, even in the case of nonlinearities involving \(u_{xx}\) with which we are primarily concerned, we are able to relax the assumption \(s>\tfrac{9}{2}\) to \(s>s_0\) where \(s_0\) is determined by the structure of \(F\) as follows.

\vspace{15pt}\begin{center}\renewcommand{\arraystretch}{1.5}
\begin{tabular}{ c | c c }
\(\mathbf{s_0}\)&\multicolumn{2}{c}{\(\mathbf{F}\)\textbf{ contains terms of the form}}\\\hline
\(\tfrac{1}{2}\)&\(u^{\alpha_0}\)&\\\hline
\(1\)&\(u^{\alpha_0}u_x\)&\(\alpha_0\geq2\)\\\hline
\multirow{2}{*}{\(\tfrac{3}{2}\)}&\(u^{\alpha_0}u_x^{\alpha_1}\)&\(\alpha_0\geq1\)\\
&\(u^{\alpha_0}u_x^{\alpha_1}u_{xx}\)&\(\alpha_0\geq2\)\\\hline
\(2\)&\(u_x^{\alpha_1}\)&\(\alpha_1\geq3\)\\\hline
\(\tfrac{5}{2}\)&\(u^{\alpha_0}u_x^{\alpha_1}u_{xx}^{\alpha_2}\)&\(\alpha_0+\alpha_1\geq2\)\\\hline
\(\tfrac{7}{2}\)&\(u^{\alpha_0}u_x^{\alpha_1}u_{xx}^{\alpha_2}\)&\(\alpha_0+\alpha_1+\alpha_2\geq3\)\\\hline
\(\tfrac{9}{2}\)&\(u_{xx}^{\alpha_2}\)&\\
\end{tabular}\end{center}\vspace{15pt}

A key difference to \cite{HG} is that by taking initial data in \(H^s\) rather than \(l^1H^s\) we do not have additional restrictions imposed by rescaling the initial data.\end{rk}


\section{Function spaces and estimates}\label{sect:spaces}


\subsection{Definitions}

We take a standard Littlewood-Paley decomposition
\[
1=\sum\limits_{j=0}^\infty S_j
\]
constructed by taking smooth \(\varphi_0\colon\mathbb{R}\rightarrow[0,1]\) such that
\[
\varphi_0(\xi)=\left\{\begin{array}{ll}1&\qquad\textrm{for }\xi\in[-1,1]\vspace{0.2cm}\\0&\qquad\textrm{for }|\xi|\geq2\end{array}\right.
\]
We then define, for \(j>0\)
\[
\varphi_j(\xi)=\varphi_0(2^{-j}\xi)-\varphi_0(2^{-j+1}\xi)
\]
and
\[
f_j=S_jf=\mathcal{F}^{-1}(\varphi_j\hat f)
\]
where \(\mathcal{F}u=\hat u\) is the spatial Fourier transform. We also use the notation
\[
f_{<j}=S_{<j}f=\sum\limits_{k<j}S_kf\qquad f_{\geq j}=S_{\geq j}f=\sum\limits_{k\geq j}S_kf
\]
Given a Fourier multiplier \(S_j\) that localises to frequencies \(\sim 2^j\) we define \(\tilde S_j\) to be a Fourier multiplier that localises to frequencies \(\sim 2^j\) and satisfies \(S_j\tilde S_j=\tilde S_j S_j=S_j\).

For each \(j\geq 0\) we take \(\mathcal{Q}_{2j}\) to be a partition of \(\R\) into intervals of length \(2^{2j}\) and
\[
1=\sum\limits_{Q\in\mathcal{Q}_{2j}}\chi_Q^2
\]
to be a smooth square partition of unity such that \(\chi_Q\sim1\) on \(Q\) and \(\supp\chi_Q\subset B\left(Q,\tfrac{1}{2}\right)\). For a Sobolev-type space \(U\) we define
\[
\|u\|_{l^2_{2j}U}^2=\sum\limits_{Q\in\mathcal{Q}_{2j}}\|u\chi_Q\|_U^2
\]
and
\[
\|u\|_{l^\infty_{2j}U}=\sup\limits_{Q\in\mathcal{Q}_{2j}}\|u\chi_Q\|_U
\]
We define the \(l^2H^s\) norm by
\[
\|u\|_{l^2H^s}^2=\sum\limits_{j\geq0}2^{2js}\|S_ju\|^2_{l^2_{2j}L^2}
\]
and note that \(\|u\|_{H^s}\sim\|u\|_{l^2H^s}\).

By replacing the partition of unity by a frequency localised version we have a Bernstein-type inequality for \(r\in[1,\infty]\) and \(1\leq p\leq q\leq\infty\)
\begin{equation}\label{est:bernstein}
\|S_ju\|_{l^2_{2j}L^r_tL^q_x}\lesssim2^{j\left(\tfrac{1}{p}-\tfrac{1}{q}\right)}\|S_ju\|_{l^2_{2j}L^r_tL^p_x}
\end{equation}

We define the local energy space \(X\) (see \cite{KPV3} Remark 3.7) with norm
\[
\|u\|_X=\sup\limits_{l\geq0}\sup\limits_{Q\in\mathcal{Q}_l}2^{-l/2}\|u\|_{L^2([0,1]\times Q)}
\]
and look for solutions in the space \(l^2X^s\subset C([0,1],H^s)\) with norm
\[
\|u\|_{l^2X^s}^2=\sum\limits_{j\geq0}2^{2js}\|S_ju\|_{l^2_{2j}X_j}^2
\]
where
\[
\|u\|_{X_j}=2^j\|u\|_X+\|u\|_{L^\infty_tL^2_x}
\]
We note that we have the low frequency estimate
\begin{equation}\label{est:LFXH}
\|S_0u\|_{l^2X^s}\lesssim\|S_0u\|_{l^2L^\infty_tH^s_x}
\end{equation}

We define the atomic space \(Y\) such that \(Y^*=X\) (see \cite{MMT} Proposition 2.1) with atoms \(a\) such that there exists some \(l\geq0\), \(Q\in\mathcal{Q}_l\) with \(\supp a\subset [0,1]\times Q\) and
\[
\|a\|_{L^2([0,1]\times Q)}\lesssim 2^{-l/2}
\]
and with norm
\[
\|f\|_{Y}=\inf\left\{\sum|c_k|:f=\sum c_ka_k,\;a_k\textrm{ atoms}\right\}
\]
We define
\[
\|f\|_{Y_j}=\inf\limits_{f=f_1+f_2}\left(2^{-j}\|f_1\|_{Y}+\|f_2\|_{L^1_tL^2_x}\right)
\]
and
\[
\|f\|_{l^2Y^s}^2=\sum\limits_{j\geq0}2^{2js}\|S_jf\|^2_{l^2_{2j}Y_j}
\]

\subsection{Bilinear estimates}

\vspace{15pt}\begin{propn}\label{propn:bil}~

a) (Algebra estimates) For \(s>\frac{1}{2}\),
\begin{equation}\label{est:Halg}
\|uv\|_{l^2H^s}\lesssim\|u\|_{l^2H^s}\|v\|_{l^2H^s}
\end{equation}
\begin{equation}\label{est:alg}
\|uv\|_{l^2X^s}\lesssim\|u\|_{l^2X^s}\|v\|_{l^2X^s}
\end{equation}

b) (Bilinear estimate) For \(\alpha+\beta>s+\tfrac{1}{2}\) and \(\alpha,\beta\geq s-1\),
\begin{equation}\label{est:bil}
\|uv\|_{l^2Y^s}\lesssim\|u\|_{l^2X^\alpha}\|v\|_{l^2X^\beta}
\end{equation}

c) (Frequency localised bilinear estimates I) For \(\alpha+\beta>s+\tfrac{1}{2}\)
\begin{equation}\label{est:bilLH}
\|S_{<j-4}uS_jv\|_{l^2Y^s}\lesssim\|u\|_{l^2X^{\alpha}}\|S_jv\|_{l^2X^\beta}\quad\beta\geq s-1
\end{equation}
\begin{equation}\label{est:bilHH}
\|S_j(S_{\geq j-4}uS_{\geq j-4}v)\|_{l^2Y^s}\lesssim2^{(s+\tfrac{1}{2}-\alpha-\beta)j}\|u\|_{l^2X^{\alpha}}\|v\|_{l^2X^\beta}\quad \alpha,\beta\geq0
\end{equation}

d) (Frequency localised bilinear estimates II) For \(s\geq0\) and \(\sigma>\tfrac{1}{2}\)
\begin{equation}\label{est:LHbilH}
\|S_{<j-4}uS_jv\|_{l^2H^s}\lesssim\|u\|_{l^2H^\sigma}\|S_jv\|_{l^2H^s}
\end{equation}
\begin{equation}\label{est:LHbilX}
\|S_{<j-4}uS_jv\|_{l^2X^s}\lesssim\|u\|_{l^2X^\sigma}\|S_jv\|_{l^2X^s}
\end{equation}
\begin{equation}\label{est:HHbilX}
\|S_j(S_{\geq j-4}uS_{\geq j-4}v)\|_{l^2X^s}\lesssim 2^{-j/2}\|u\|_{l^2X^\sigma}\|v\|_{l^2X^\sigma}
\end{equation}
\begin{equation}\label{est:LHbilY}
\|S_{<j-4}uS_jv\|_{l^2Y^s}\lesssim\|u\|_{l^2X^\sigma}\|S_jv\|_{l^2Y^s}
\end{equation}

\end{propn}
\begin{proof}

a) The estimate \eqref{est:Halg} follows from the fact that the \(H^s\) and \(l^2H^s\) norms are equivalent.

For \eqref{est:alg} we use the Littlewood-Paley trichotomy and consider terms of the form \(S_k(S_iuS_jv)\).

\textbf{High-low interactions.} \(|i-k|< 4\), \(j<k-4\). Using Bernstein's inequality \eqref{est:bernstein} we have
\begin{align*}
\|S_k(S_iuS_jv)\|_{l^2_{2k}X_k}&\lesssim\|S_iu\|_{l^2_{2i}X_i}\|S_jv\|_{L^{\infty}_{t,x}}\\
&\lesssim 2^{j/2}\|S_iu\|_{l^2_{2i}X_i}\|S_jv\|_{L^\infty_tL^2_x}
\end{align*}
The symmetric low-high interaction is similar.

\textbf{High-high interactions.} \(|i-j|\leq 4\), \(i,j\geq k-4\). Using Bernstein's inequality \eqref{est:bernstein}, Cauchy-Schwarz and switching interval size we have
\begin{align*}
\|S_k(S_iuS_jv)\|_{l^2_{2k}X_k}&\lesssim2^{k/2}\|S_iu\|_{l^2_{2k}X_k}\|S_jv\|_{L^\infty_tL^2_x}\\
&\lesssim2^{i-k/2}\|S_iu\|_{l^2_{2i}X_i}\|S_jv\|_{L^\infty_tL^2_x}
\end{align*}

b) We note that for all \(j\)
\[
\|f\|_{l^2_{2k}Y_k}\lesssim\|f\|_{l^2_{2j}L^2_{t,x}}
\]

\textbf{High-low interactions.} \(|i-k|< 4\), \(j<k-4\).
\begin{align*}
\|S_k(S_iuS_jv)\|_{l^2_{2k}Y_k}&\lesssim\|S_iu\|_{l^\infty_{2j}L^2_{t,x}}\|S_jv\|_{l^2_{2j}L^\infty_{t,x}}\\
&\lesssim2^{\tfrac{3}{2}j}2^{-i}\|S_iu\|_{X_i}\|S_jv\|_{l^2_{2j}L^\infty_tL^2_x}
\end{align*}
The symmetric low-high interaction is similar.

\textbf{High-high interactions.} \(|i-j|\leq 4\), \(i,j\geq k-4\).
\begin{align*}
\|S_k(S_iuS_jv)\|_{l^2_{2k}Y_k}&\lesssim\|S_k(S_iuS_jv)\|_{l^2_{2j}L^2_{t,x}}\\
&\lesssim2^{k/2}\|S_iuS_jv\|_{l^2_{2j}L^2_tL^1_x}\\
&\lesssim2^{k/2}\|S_iu\|_{l^\infty_{2j}L^2_{t,x}}\|S_jv\|_{l^2_{2j}L^\infty_tL^2_x}\\
&\lesssim2^{k/2}\|S_iu\|_{X_i}\|S_jv\|_{l^2_{2j}L^\infty_tL^2_x}
\end{align*}

c) The estimates \eqref{est:bilLH}, \eqref{est:bilHH} follow from the proof of part (b).

d) For \eqref{est:LHbilH} we take \(i<k-4\), \(|j-k|<4\) and consider
\begin{align*}
\|S_k(S_iuS_jv)\|_{l^2_{2k}L^2}&\lesssim\|S_iu\|_{L^\infty}\|S_jv\|_{l^2_{2k}L^2}\\
&\lesssim2^{i/2}\|S_iu\|_{L^2}\|S_jv\|_{l^2_{2k}L^2}
\end{align*}
The estimates \eqref{est:LHbilX}, \eqref{est:LHbilY} are identical. The estimate \eqref{est:HHbilX} follows from the proof of \eqref{est:alg}.

\end{proof}
\vspace{15pt}

As a consequence of the algebra estimates \eqref{est:Halg}, \eqref{est:alg} and the bilinear estimates \eqref{est:LHbilH}, \eqref{est:LHbilX} and \eqref{est:LHbilY}, we have the following corollary.

\vspace{15pt}\begin{cor}\label{cor:exptype}~

a) (Estimates with an exponential) For \(s>\tfrac{1}{2}\)
\begin{equation}\label{est:expH}
\|e^au\|_{l^2H^s}\leq e^{C\|a\|_{l^2H^s}}\|u\|_{l^2H^s}
\end{equation}
\begin{equation}\label{est:expX}
\|e^au\|_{l^2X^s}\leq e^{C\|a\|_{l^2X^s}}\|u\|_{l^2X^s}
\end{equation}

b) (Frequency localised estimates with an exponential) For \(s\geq0\) and \(\sigma>\tfrac{1}{2}\)
\begin{equation}\label{est:FLexpH}
\|S_{<j-4}(e^a)S_ju\|_{l^2H^s}\leq e^{C\|a\|_{l^2H^\sigma}}\|S_ju\|_{l^2H^s}
\end{equation}
\begin{equation}\label{est:FLexpX}
\|S_{<j-4}(e^a)S_ju\|_{l^2X^s}\leq e^{C\|a\|_{l^2X^\sigma}}\|S_ju\|_{l^2X^s}
\end{equation}
\begin{equation}\label{est:FLexpY}
\|S_{<j-4}(e^a)S_jf\|_{l^2Y^s}\leq e^{C\|a\|_{l^2X^\sigma}}\|S_jf\|_{l^2Y^s}
\end{equation}

\end{cor}

\subsection{Trilinear estimates}

\vspace{15pt}\begin{propn}\label{propn:tri} Suppose \(\alpha+\beta+\gamma>s+1\).

a) (Trilinear Estimate) If \(\alpha,\beta,\gamma\geq s-2\) and \(\alpha+\beta,\beta+\gamma,\gamma+\alpha>s-\frac{1}{2}\) then,
\begin{equation}\label{est:tri}
\|uvw\|_{l^2Y^s}\lesssim\|u\|_{l^2X^\alpha}\|v\|_{l^2X^\beta}\|w\|_{l^2X^\gamma}
\end{equation}

b) (Frequency Localised Trilinear Estimate) If \(j\leq k\), \(\gamma\geq s-2\) and \(\beta+\gamma>s-\frac{1}{2}\) then,
\begin{equation}\label{est:triLHH}
\|S_{<j-4}uS_jvS_kw\|_{l^2Y^s}\lesssim\|u\|_{l^2X^\alpha}\|S_jv\|_{l^2X^\beta}\|S_kw\|_{l^2X^\gamma}
\end{equation}
\end{propn}\vspace{15pt}
The proof relies on the following lemma.

\vspace{15pt}\begin{lem}\label{lem:tri-est}
Suppose \(i\leq j\leq k\leq l\) then
\begin{equation}
\int_{[0,1]}\int_\R u_iv_jw_kz_l\,dxdt\lesssim2^{\tfrac{3}{2}i+\tfrac{3}{2}j-k-l}\|u_i\|_{l^2_{2i}X_i}\|v_j\|_{l^2_{2j}X_j}\|w_k\|_{l^2_{2k}X_k}\|z_l\|_{l^2_{2l}X_l}
\end{equation}
\end{lem}

\begin{proof}We have
\begin{align*}
\iint u_iv_jw_kz_l\,dxdt&=\sum\limits_{Q\in\mathcal{Q}_{2i}}\iint u_i\chi_Qv_j\chi_Qw_kz_l\,dxdt\\
&\lesssim\left(\sum\limits_{Q\in\mathcal{Q}_{2i}}\|u_i\chi_Q\|_{L^\infty_{t,x}}\|v_j\chi_Q\|_{L^\infty_{t,x}}\right)\|w_k\|_{l^\infty_{2i} L^2_{t,x}}\|z_l\|_{l^\infty_{2i}L^2_{t,x}}\\
&\lesssim\|u_i\|_{l^2_{2i}L^\infty_{t,x}}\|v_j\|_{l^2_{2i}L^\infty_{t,x}}\|w_k\|_{l^\infty_{2i} L^2_{t,x}}\|z_l\|_{l^\infty_{2i}L^2_{t,x}}\\
&\lesssim2^{\tfrac{5}{2}i+\tfrac{1}{2}j-k-l}\|u_i\|_{l^2_{2i}L^\infty_tL^2_x}\|v_j\|_{l^2_{2i}L^\infty_tL^2_x}\|w_k\|_{X_k}\|z_l\|_{X_l}\\
&\lesssim2^{\tfrac{3}{2}i+\tfrac{3}{2}j-k-l}\|u_i\|_{l^2_{2i}L^\infty_tL^2_x}\|v_j\|_{l^2_{2j}L^\infty_tL^2_x}\|w_k\|_{l^2_{2k}X_k}\|z_l\|_{l^2_{2l}X_l}\end{align*}

\end{proof}\vspace{15pt}

To complete the proof of \eqref{est:tri} we use that with respect to \(L^2\) duality \((l^2_{2l}Y_l)^*=l^2_{2l}X_l\). We consider terms of the form \(S_l(S_iuS_jvS_kw)\) and by symmetry we may assume \(i\leq j\leq k\). The non-zero interactions can be divided into the following cases.

\textbf{Case 1.} \(|l-k|<4\). We use the above estimate and the fact that \(\gamma\geq s-2\) to get
\[
\|S_l(S_iuS_jvS_kw)\|_{l^2Y^s}\lesssim2^{(\tfrac{3}{2}-\alpha)i}2^{(s-\tfrac{1}{2}-\beta-\gamma)j}\|S_iu\|_{l^2X^\alpha}\|S_ju\|_{l^2X^\beta}\|S_ku\|_{l^2X^\gamma}
\]

If \(\alpha\geq\tfrac{3}{2}\) then we use that \(\beta+\gamma>s-\tfrac{1}{2}\). If \(\alpha<\tfrac{3}{2}\) then summing over \(i\leq j\) we have
\[
\sum\limits_{i\leq j}\|S_l(S_iuS_jvS_kw)\|_{l^2Y^s}\lesssim2^{(s+1-\alpha-\beta-\gamma)j}\|u\|_{l^2X^\alpha}\|S_ju\|_{l^2X^\beta}\|S_ku\|_{l^2X^\gamma}
\]

\textbf{Case 2.} \(k-j<4\), \(i\leq l\leq k-4\). Switching the roles of \(z_l\) and \(v_j\) in Lemma \ref{lem:tri-est} we have
\[
\|S_l(S_iuS_jvS_kw)\|_{l^2Y^s}\lesssim2^{(\tfrac{3}{2}-\alpha)i}2^{(s-\tfrac{1}{2}-\beta-\gamma)l}\|S_iu\|_{l^2X^\alpha}\|S_jv\|_{l^2X^\beta}\|S_kw\|_{l^2X^\gamma}
\]
and a similar argument to Case 1 gives that the sum converges.

\textbf{Case 3.} \(k-j<4\), \(l\leq i\leq k-4\). Switching the roles of \(z_l\) and \(u_i\) in Lemma \ref{lem:tri-est}
\[
\|S_l(S_iuS_jvS_kw)\|_{l^2Y^s}\lesssim2^{(s+\tfrac{3}{2})l}2^{-(\frac{1}{2}+\alpha+\beta+\gamma)i}\|S_iu\|_{l^2X^\alpha}\|S_jv\|_{l^2X^\beta}\|S_kw\|_{l^2X^\gamma}
\]so summing over \(i\geq l\)
\[
\sum\limits_{i\geq l}\|S_l(S_iuS_jvS_kw)\|_{l^2Y^s}\lesssim2^{(s+1-\alpha-\beta-\gamma)l}\|u\|_{l^2X^\alpha}\|S_jv\|_{l^2X^\beta}\|S_kw\|_{l^2X^\gamma}
\]

The frequency localised estimate \eqref{est:triLHH} follows from the proof of \eqref{est:tri}.

\vspace{15pt}


\subsection{Commutator estimates}
\begin{propn}For \(s\geq0\) and \(\sigma>\tfrac{7}{2}\), we have the estimate
\begin{equation}\label{est:com}
\|[S_j,\partial_xa_{<j-4}]\partial_x^2u\|_{l^2Y^s}\lesssim\|\partial_xa\|_{l^2X^{\sigma-1}}\|\tilde S_ju\|_{l^2X^s}
\end{equation}
\end{propn}
\begin{proof}
Due to the frequency localisation we can replace \(u\) by \(\tilde S_ju\) and write
\[
[S_j,\partial_xa_{<j-4}]\partial_x^2\tilde S_ju=\partial_x[S_j,\partial_xa_{<j-4}]\partial_x\tilde S_ju-[S_j,\partial_x^2a_{<j-4}]\partial_x\tilde S_ju
\]
The second commutator can be estimated by the bilinear estimate \eqref{est:bilLH}. To estimate the first commutator, as in \cite{MMT} Proposition 3.2, we can write
\[
\partial_x[S_j,\partial_xa_{<j-4}]\partial_x\tilde S_ju=C(\partial_xa_{<j-4},\partial_x\tilde S_ju)
\]
for a disposable operator
\[
C(f,g)=\iint f(x+y)g(x+z)w(y,z)\,dydz\qquad\|w\|_{L^1}\lesssim1
\]
and then use the bilinear estimate \eqref{est:bilLH}.

\end{proof}

\section{Linear estimates}\label{sect:lin}

\subsection{The Airy equation}
We consider the linear equation
\begin{equation}\label{eq:linear}
\pde{(\partial_t+\partial_x^3)u=f}{u(0)=u_0}
\end{equation}
An identical argument to \cite{HG} Proposition 4.3 gives the following result.

\vspace{15pt}\begin{propn}\label{propn:linLocWP}
For \(s\geq 0\), the linear equation \eqref{eq:linear} is well-posed in \(H^s\) on the time interval \([0,1]\) with the estimate
\begin{equation}\label{est:mainlin}
\|u\|_{l^2X^s}\lesssim\|u_0\|_{l^2H^s}+\|f\|_{l^2Y^s}
\end{equation}\end{propn}\vspace{15pt}

\subsection{A paradifferential equation}
We define the paraproduct
\[
T_au=\sum\limits_{j\geq0}S_{<j-4}aS_ju
\]
and consider the equation
\begin{equation}\label{eq:paralinearpde}
\pde{(\partial_t+\partial_x^3-T_{\partial_xa}\partial_x^2)u=f}{u(0)=u_0}
\end{equation}
The equation \eqref{eq:paralinearpde} is a system of equations for the frequency localised components \(u_j\). We construct approximate solutions to each of these frequency localised equations by conjugating the linear operator by a suitable exponential term, similar to \cite{BT}. We then use these approximate solutions to construct a solution to \eqref{eq:paralinearpde}. The main result we prove is the following.

\vspace{15pt}\begin{propn}\label{propn:modairyparaEST}
Let \(s\geq0\), \(\sigma>\tfrac{7}{2}\) and \(a\in l^2X^\sigma\). Then there exists \(\delta=\delta(s,\sigma,\|a\|_{l^2X^\sigma})>0\) such that if \(a\) satisfies the estimates

\begin{equation}\label{cond:para1.1}
\|\partial_xa\|_{l^2X^{\sigma-1}}\leq\delta
\end{equation}
and
\begin{equation}
\label{cond:para2}\|T_{(\partial_t+\partial_x^3)a}\|_{l^2X^s\rightarrow l^2Y^s}\leq\delta
\end{equation}
there exists a unique solution \(u\) to \eqref{eq:paralinearpde} that satisfies the estimate
\begin{equation}\label{est:paralinearpde}
\|u\|_{l^2X^s}\lesssim e^{C\|a\|_{l^2X^\sigma}}\left(\|u_0\|_{l^2H^s}+\|f\|_{l^2Y^s}\right)
\end{equation}\end{propn}\vspace{15pt}

\subsubsection{Proof of existence.}

We start by looking for a frequency localised solution to the equation
\begin{equation}\label{eq:paralinearpdefreqloc}
\pde{(\partial_t+\partial_x^3-\partial_xa_{<j-4}\partial_x^2)u_j=f_j}{u_j(0)=u_{0j}}
\end{equation}

\vspace{15pt}\begin{lem}\label{lem:freclocapprox}
If \(s\geq0\), \(\sigma>\tfrac{7}{2}\) and \(a\in l^2X^\sigma\) satisfies \eqref{cond:para1.1} and \eqref{cond:para2} for \(\delta>0\) sufficiently small, the equation \eqref{eq:paralinearpdefreqloc} has a frequency localised approximate solution \(\tilde u_j\) satisfying the estimate
\begin{equation}\label{est:ujtildebd}
\|\tilde u_j\|_{l^2X^s}\lesssim e^{C\|a\|_{l^2X^\sigma}}\left(\|u_{0j}\|_{l^2H^s}+\|f_j\|_{l^2Y^s}\right)
\end{equation}
and the error estimates
\begin{equation}\label{est:ujtilde0error}
\|\tilde u_j(0)-u_{0j}\|_{l^2H^s}\lesssim\delta e^{C\|a\|_{l^2X^\sigma}}\|u_{0j}\|_{l^2H^s}
\end{equation}
\begin{equation}\label{est:ujtildeerror}
\|(\partial_t+\partial_x^3-\partial_xa_{<j-4}\partial_x^2)\tilde u_j-f_j\|_{l^2Y^s}\lesssim\delta e^{C\|a\|_{l^2X^\sigma}}\left(\|u_{0j}\|_{l^2H^s}+\|f_j\|_{l^2Y^s}\right)
\end{equation}
\end{lem}

\begin{proof}~

Using Proposition \ref{propn:linLocWP} we take \(v_j\) to be the solution to
\begin{equation}\label{eq:normalisedfreqlocpara}
\pde{(\partial_t+\partial_x^3)v_j=S_{<j-4}(e^{-\frac{1}{3}a_{<j-4}})f_j}{v_j(0)=S_{<j-4}(e^{-\frac{1}{3}a_{<j-4}(0)})u_{0j}}
\end{equation}
so by \eqref{est:mainlin} and Corollary \ref{cor:exptype}
\begin{align*}
\|v_j\|_{l^2X^s}&\lesssim\|S_{<j-4}(e^{-\frac{1}{3}a_{<j-4}(0)})u_{0j}\|_{l^2H^s}+\|S_{<j-4}(e^{-\frac{1}{3}a_{<j-4}})f_j\|_{l^2Y^s}\\
&\lesssim e^{C\|a\|_{l^2X^\sigma}}\left(\|u_{0j}\|_{l^2H^s}+\|f_j\|_{l^2Y^s}\right)
\end{align*}
We construct a frequency localised approximate solution by taking
\[
\tilde u_j=\tilde S_j\left(e^{\frac{1}{3}a_{<j-4}}v_j\right)
\]
The estimate \eqref{est:FLexpH} then gives
\begin{align*}
\|\tilde u_j\|_{l^2X^s}&\lesssim e^{C\|a\|_{l^2X^\sigma}}\|v_j\|_{l^2X^s}\\
&\lesssim e^{C\|a\|_{l^2X^\sigma}}\left(\|u_{0j}\|_{l^2H^s}+\|f_j\|_{l^2Y^s}\right)
\end{align*}
proving \eqref{est:ujtildebd}.

For the error estimate \eqref{est:ujtilde0error} we have
\begin{align*}
\tilde u_j(0)&=\tilde S_j\left(e^{\frac{1}{3}a_{<j-4}(0)}S_{<j-4}(e^{-\frac{1}{3}a_{<j-4}(0)})u_{0j}\right)\\
&=u_{0j}-\tilde S_j\left(e^{\frac{1}{3}a_{<j-4}(0)}S_{\geq j-4}(e^{-\frac{1}{3}a_{<j-4}(0)})u_{0j}\right)
\end{align*}
We can then use the bilinear estimates \eqref{est:LHbilH}, \eqref{est:expH} to get
\begin{align*}
\|\tilde u_j(0)-u_{0j}\|_{l^2H^s}&\lesssim\|e^{\frac{1}{3}a_{<j-4}(0)}S_{\geq j-4}(e^{-\frac{1}{3}a_{<j-4}(0)})\|_{l^2H^\sigma}\|u_{0j}\|_{l^2H^s}\\
&\lesssim \|S_{\geq j-4}(e^{-\frac{1}{3}a_{<j-4}})\|_{l^2X^\sigma}e^{C\|a\|_{l^2X^\sigma}}\|u_{0j}\|_{l^2H^s}
\end{align*}
Finally we note that
\begin{align*}
\|S_{\geq j-4}(e^{-\frac{1}{3}a_{<j-4}})\|_{l^2X^\sigma}&\lesssim\|\partial_xS_{\geq j-4}(e^{-\frac{1}{3}a_{<j-4}})\|_{l^2X^{\sigma-1}}\\
&\lesssim\|\partial_xa_{<j-4}\|_{l^2X^{\sigma-1}}e^{C\|a\|_{l^2X^\sigma}}
\end{align*}

For \eqref{est:ujtildeerror} we calculate
\begin{align*}
(\partial_t+\partial_x^3-\partial_xa_{<j-4}\partial_x^2)\tilde u_j=\;&\tilde S_j(e^{\frac{1}{3}a_{<j-4}}S_{<j-4}(e^{-\frac{1}{3}a_{<j-4}})f_j)\\
&+[\tilde S_j,\partial_xa_{<j-4}]\partial_x^2(e^{\frac{1}{3}a_{<j-4}}v_j)\\
&+\tilde S_jR(a_{<j-4},e^{\frac{1}{3}a_{<j-4}}v_j)
\end{align*}
where
\begin{equation}\label{defn:R}
R(g,h)=\left(\tfrac{1}{3}(\partial_t+\partial_x^3)g-\tfrac{1}{3}\partial_xg\partial_x^2g+\tfrac{1}{27}(\partial_xg)^3\right)h+\left(\partial_x^2g-\tfrac{1}{3}(\partial_xg)^2\right)\partial_xh
\end{equation}
As in the estimate \eqref{est:ujtilde0error} we have
\begin{align*}
\|\tilde S_j(e^{\frac{1}{3}a_{<j-4}}S_{<j-4}(e^{-\frac{1}{3}a_{<j-4}})&f_j)-f_j\|_{l^2Y^s}\\
&\lesssim\|e^{\frac{1}{3}a_{<j-4}}S_{\geq j-4}(e^{-\frac{1}{3}a_{<j-4}})\|_{l^2X^\sigma}\|f_j\|_{l^2Y^s}\\
&\lesssim\delta e^{C\|a\|_{l^2X^\sigma}}\|f_j\|_{l^2Y^s}
\end{align*}
From the commutator estimate \eqref{est:com} we have
\[
\|[\tilde S_j,\partial_xa_{<j-4}]\partial_x^2(e^{\frac{1}{3}a_{<j-4}}v_j)\|_{l^2Y^s}\lesssim\|\partial_xa\|_{l^2X^{\sigma-1}}\|\tilde{\tilde S}_j(e^{\frac{1}{3}a_{<j-4}}v_j)\|_{l^2X^s}
\]
To estimate the remainder term we write
\[
\tilde S_jR(a_{<j-4},e^{\frac{1}{3}a_{<j-4}}v_j)=\tilde S_jR(a_{<j-4},\tilde{\tilde S}_j(e^{\frac{1}{3}a_{<j-4}}v_j))
\]
The hypothesis \eqref{cond:para2} gives
\[
\|(\partial_t+\partial_x^3)a_{<j-4}\tilde{\tilde S}_j(e^{\frac{1}{3}a_{<j-4}}v_j)\|_{l^2Y^s}\lesssim\delta\|\tilde{\tilde S}_j(e^{\frac{1}{3}a_{<j-4}}v_j)\|_{l^2X^s}
\]
The remaining terms can be estimated using Propositions \ref{propn:bil} and Corollary \ref{cor:exptype} with the hypothesis \eqref{cond:para1.1} to get
\begin{align*}
\|\tilde S_jR(a_{<j-4},e^{\frac{1}{3}a_{<j-4}}v_j)\|_{l^2Y^s}&\lesssim\delta\|\tilde{\tilde S}_j(e^{\frac{1}{3}a_{<j-4}}v_j)\|_{l^2X^s}\\
&\lesssim\delta e^{C\|a\|_{l^2X^\sigma}}\|v_j\|_{l^2X^s}
\end{align*}
The worst term in this estimate is \(\partial_x^2a_{<j-4}\partial_x\tilde{\tilde S}_j(e^{\frac{1}{3}a_{<j-4}}v_j)\) which requires \(\sigma>\tfrac{7}{2}\).
\end{proof}
\vspace{15pt}

Taking \(f_j^{(0)}=f_j\) and \(u_{0j}^{(0)}=u_{0j}\) we can use Lemma \ref{lem:freclocapprox} find an approximate solution \(\tilde u_j^{(0)}\) to \eqref{eq:paralinearpdefreqloc}. We then take
\[
f_j^{(n+1)}=f_j^{(n)}-(\partial_t+\partial_x^3-\partial_xa_{<j-4}\partial_x^2)\tilde u_j^{(n)}\qquad u_{0j}^{(n+1)}=u_{0j}^{(n)}-\tilde u_j^{(n)}(0)
\]
and use Lemma \ref{lem:freclocapprox} to construct a sequence of approximate solutions \(\tilde u_j^{(n)}\). For \(\delta\) sufficiently small
\begin{equation}\label{defn:constructeduj}
u_j=\sum\limits_{n\geq0}\tilde u_j^{(n)}
\end{equation}
converges to a solution to \eqref{eq:paralinearpdefreqloc} in \(l^2X^s\) satisfying the estimate
\begin{equation}
\label{est:ujbd}\|u_j\|_{l^2X^s}\lesssim e^{C\|a\|_{l^2X^\sigma}}\left(\|u_{0j}\|_{l^2H^s}+\|f_j\|_{l^2Y^s}\right)
\end{equation}
We can now use the solutions to \eqref{eq:paralinearpdefreqloc} to construct an approximate solution to \eqref{eq:paralinearpde}.

\vspace{15pt}\begin{lem}\label{lem:paraLocWP}
If \(s\geq0,\sigma>\tfrac{7}{2}\) and \(a\in l^2X^\sigma\) satisfies \eqref{cond:para1.1} and \eqref{cond:para2} for \(\delta>0\) sufficiently small, there exists an approximate solution \(\tilde u\) to \eqref{eq:paralinearpde} satisfying the estimate
\begin{equation}\label{est:apparalinearpde}
\|\tilde u\|_{l^2X^s}\lesssim e^{C\|a\|_{l^2X^\sigma}}\left(\|u_0\|_{l^2H^s}+\|f\|_{l^2Y^s}\right)
\end{equation}
and the error estimate
\begin{equation}\label{est:paralinearpdeerror}
\|(\partial_t+\partial_x^3-T_{\partial_xa}\partial_x^2)\tilde u-f\|_{l^2X^s}\lesssim\delta e^{C\|a\|_{l^2X^\sigma}}\left(\|u_0\|_{l^2H^s}+\|f\|_{l^2Y^s}\right)
\end{equation}
\end{lem}

\begin{proof}
We define
\[
\tilde u=\sum\limits_{j\geq0}u_j
\]
where \(u_j\) is a solution to \eqref{eq:paralinearpdefreqloc}. Due to the frequency localisation of \(u_j,u_{0j},f_j\) and and the estimate \eqref{est:ujbd} this converges in \(l^2X^s\) and satisfies the estimate \eqref{est:apparalinearpde}. We can calculate the error
\begin{align*}
(\partial_t+\partial_x^3-T_{\partial_xa}\partial_x^2)\tilde u-f=\sum\limits_{j\geq0}\sum\limits_{k\sim j}&\left\{S_j(\partial_x(a_{<k-4}-a_{<j-4})\partial_x^2u_k)\right.\\
&\left.+[S_j,\partial_xa_{<j-4}]\partial_x^2u_k\right\}
\end{align*}
The commutator can be estimated using \eqref{est:com} to get
\[
\|[S_j,\partial_xa_{<j-4}]\partial_x^2u_k\|_{l^2Y^s}\lesssim\|\partial_xa\|_{l^2X^{\sigma-1}}\|\tilde S_ju_k\|_{l^2X^s}
\]
For the remaining term we can use the bilinear estimate \eqref{est:bilHH} and the frequency localisation to get
\[
\|S_j(\partial_x(a_{<k-4}-a_{<j-4})\partial_x^2u_k)\|_{l^2Y^s}\lesssim\|\partial_xa\|_{l^2X^{\sigma-1}}\|u_k\|_{l^2X^s}
\]
So from \eqref{est:ujbd} and almost orthogonality we have
\begin{align*}
\|(\partial_t+\partial_x^3-T_{\partial_xa}\partial_x^2)\tilde u-f\|_{l^2Y^s}&\lesssim\left(\sum\limits_{j\geq0}\sum\limits_{k\sim j}\|\partial_xa\|_{l^2X^{\sigma-1}}^2\|u_k\|_{l^2X^s}^2\right)^{\frac{1}{2}}\\
&\lesssim\delta e^{C\|a\|_{l^2X^\sigma}}\left(\|u_0\|_{l^2H^s}+\|f\|_{l^2Y^s}\right)
\end{align*}
\end{proof}\vspace{15pt}

\subsubsection{Proof of uniqueness.}

The uniqueness of the solution to \eqref{eq:paralinearpde} is a corollary of the following proposition.

\vspace{15pt}
\begin{propn}\label{propn:modAiryparadiffs}
Suppose \(s\geq0,\sigma>\tfrac{7}{2}\) and for \(i=1,2\), \(a\sk\in l^2X^\sigma\) satisfy \eqref{cond:para1.1} and \eqref{cond:para2}. If \(u\sk\) solves
\begin{equation}
\pde{(\partial_t+\partial_x^3-T_{\partial_xa\sk}\partial_x^2)u\sk=f\sk}{u\sk(0)=u_0\sk}
\end{equation}
then for \(\delta=\delta(s,\sigma,\|a\so\|_{l^2X^\sigma},\|a\st\|_{l^2X^\sigma})>0\) sufficiently small,
\begin{align}\label{est:uskdif}
&\|u\so-u\st\|_{l^2X^s}\\
&\lesssim C(\|a\so\|_{l^2X^\sigma},\|a\st\|_{l^2X^\sigma})\left(\|u_0\so-u_0\st\|_{l^2H^s}+\|f\so-f\st\|_{l^2Y^s}\right.\notag\\
&\left.+(\|a\so\!\!-\!a\st\|_{l^2X^\sigma}\!\!+\!\|T_{(\partial_t+\partial_x^3)(a\so\!-a\st)}\|_{l^2X^s\rightarrow l^2Y^s}\!)(\|u_0\so\|_{l^2H^s}\!+\!\|f\so\|_{l^2Y^s}\!)\!\right)\notag
\end{align}
\end{propn}

\begin{proof}
We define
\[
v_j\sk=\tilde S_j(e^{-\frac{1}{3}a_{<j-4}\sk}S_ju\sk)
\]
Then,
\[
\pde{(\partial_t+\partial_x^3)v_j\sk=\tilde S_j(e^{-\tfrac{1}{3}a_{<j-4}\sk}f_j\sk)-r(a_{<j-4}\sk,v_j\sk)}{v_j\sk(0)=\tilde S_j(e^{-\frac{1}{3}a_{<j-4}\sk(0)}u_{0j}\sk)}
\]
where
\begin{align*}
r(a_{<j-4},v_j)=&\left(\tfrac{1}{3}(\partial_t+\partial_x^3)a_{<j-4}+\tfrac{2}{27}(\partial_xa_{<j-4})^3\right)v_j\\
&+\left(\partial_x^2a_{<j-4}-\tfrac{1}{3}(\partial_xa_{<j-4})^2\right)\partial_xv_j
\end{align*}

From the estimate \eqref{est:mainlin} with the estimates of Corollary \ref{cor:exptype} we have
\[
\|v_j\sk\|_{l^2X^s}\lesssim e^{C\|a\sk\|_{l^2X^\sigma}}(\|u_{0j}\sk\|_{l^2H^s}+\|f_j\sk\|_{l^2Y^s})+\|r(a_{<j-4}\sk,v_j\sk)\|_{l^2Y^s}
\]

From the hypothesis \eqref{cond:para2} we have
\[
\|(\partial_t+\partial_x^3)a_{<j-4}\sk v_j\sk\|_{l^2Y^s}\lesssim\delta\|v_j\sk\|_{l^2X^s}
\]
and from the hypothesis \eqref{cond:para1.1} and the estimates of Proposition \ref{propn:bil} we have
\[
\|r(a_{<j-4}\sk,v_j\sk)\|_{l^2X^s}\lesssim\delta\|v_j\sk\|_{l^2X^s}
\]
We note that the worst term we need to estimate is
\[
\|\partial_x^2a_{<j-4}\sk\partial_xv_j\sk\|_{l^2Y^s}\lesssim\|\partial_xa\sk\|_{l^2X^{\sigma-1}}\|v_j\|_{l^2X^s}
\]
which requires \(\sigma>\tfrac{7}{2}\). For \(\delta\) sufficiently small we then have
\begin{equation}\label{est:vibd}
\|v_j\sk\|_{l^2X^s}\lesssim C(\|a\sk\|_{l^2X^\sigma})(\|u_{0j}\sk\|_{l^2H^s}+\|f_j\sk\|_{l^2Y^s})
\end{equation}

The difference \(v_j\so-v_j\st\) satisfies
\[
\left\{\begin{array}{rl}(\partial_t+\partial_x^3)(v_j\so-v_j\st)=&\!\!\!\!\tilde S_j(e^{-\tfrac{1}{3}a_{<j-4}\so}f_j\so-e^{-\tfrac{1}{3}a_{<j-4}\st}f_j\st)\vspace{0.2cm}\\&\!\!\!\!+r(a_{<j-4}\so,v_j\so)-r(a_{<j-4}\st,v_j\st)\vspace{0.2cm}\\
v_j\so(0)-v_j\st(0)=&\!\!\!\!\tilde S_j(e^{-\frac{1}{3}a_{<j-4}\so(0)}u_{0j}\so-e^{-\frac{1}{3}a_{<j-4}\st(0)}u_{0j}\st)\end{array}\right.
\]
We can use \eqref{est:alg} and \eqref{est:FLexpY} to estimate
\begin{align*}
\|\tilde S_j(e^{-\tfrac{1}{3}a_{<j-4}\so}f_j\so&-e^{-\tfrac{1}{3}a_{<j-4}\st}f_j\st)\|_{l^2Y^s}\\
\lesssim&\;e^{C\max(\|a\so\|_{l^2X^\sigma},\|a\st\|_{l^2X^\sigma})}\|a\so-a\st\|_{l^2X^\sigma}\|f_j\so\|_{l^2Y^s}\\
&\;+e^{C\|a\st\|_{l^2X^\sigma}}\|f_j\so-f_j\st\|_{l^2Y^s}
\end{align*}
and similarly using \eqref{est:Halg} and \eqref{est:alg}
\begin{align*}
\|\tilde S_j(e^{-\tfrac{1}{3}a_{<j-4}\so}u_{0j}\so&-e^{-\tfrac{1}{3}a_{<j-4}\st}u_{0j}\st)\|_{l^2H^s}\\
\lesssim&\;e^{C\max(\|a\so\|_{l^2X^\sigma},\|a\st\|_{l^2X^\sigma})}\|a\so-a\st\|_{l^2X^\sigma}\|u_{0j}\so\|_{l^2H^s}\\
&\;+e^{C\|a\st\|_{l^2X^\sigma}}\|u_{0j}\so-u_{0j}\st\|_{l^2H^s}
\end{align*}
From the hypothesis \eqref{cond:para2} we have
\begin{align*}
\|(\partial_t+\partial_x^3)a_{<j-4}\so v_j\so&-(\partial_t+\partial_x^3)a_{<j-4}\st v_j\st\|_{l^2Y^s}\\
\lesssim&\;\|T_{(\partial_t+\partial_x^3)(a\so-a\st)}\|_{l^2X^s\rightarrow l^2Y^s}\|v_j\so\|_{l^2X^s}\\
&\;+\delta\|v_j\so-v_j\st\|_{l^2X^s}
\end{align*}
For the remaining terms we use the hypothesis \eqref{cond:para1.1} with the estimates \eqref{est:alg}, \eqref{est:bilLH} and \eqref{est:tri} to get
\begin{align*}
\|r(a_{<j-4}\so,v_j\so)&-r(a_{<j-4}\st,v_j\st)\|_{l^2Y^s}\\
\lesssim&\; C(\|a\so\|_{l^2X^\sigma},\|a\st\|_{l^2X^\sigma})\left(\|a\so-a\st\|_{l^2X^\sigma}\right.\\
&\left.+\|T_{(\partial_t+\partial_x^3)(a\so-a\st)}\|_{l^2X^s\rightarrow l^2Y^s}\right)\|v_j\so\|_{l^2X^s}\\
&\left.+\delta\|v_j\so-v_j\st\|_{l^2X^s}\right.
\end{align*}

So, choosing \(\delta\) sufficiently small, from \eqref{est:mainlin} and \eqref{est:vibd} we have
\begin{align*}
&\|v_j\so-v_j\st\|_{l^2X^s}\\
&\lesssim C(\|a\so\|_{l^2X^\sigma},\|a\st\|_{l^2X^\sigma})\left(\|u_{0j}\so-u_{0j}\st\|_{l^2H^s}+\|f_j\so-f_j\st\|_{l^2Y^s}\right.\notag\\
&\left.+(\|a\so\!\!-\!a\st\|_{l^2X^\sigma}\!\!+\!\|T_{(\partial_t+\partial_x^3)(a\so\!-a\st)}\|_{l^2X^s\rightarrow l^2Y^s}\!)(\|u_{0j}\so\|_{l^2H^s}\!+\!\|f_j\so\|_{l^2Y^s}\!)\!\right)\notag
\end{align*}

If we write
\[
u_j\sk=S_{<j-4}(e^{\frac{1}{3}a_{<j-4}\sk})w_j\sk
\]
then as for the estimate \eqref{est:ujtilde0error}, we have
\[
\|v_j\sk-w_j\sk\|_{l^2X^s}\lesssim\delta e^{C\|a\|_{l^2X^\sigma}}\|w_j\sk\|_{l^2X^s}
\]
So, for sufficiently small \(\delta\),
\[
\|w_j\sk\|_{l^2X^s}\lesssim C(\|a\sk\|_{l^2X^\sigma})\|v_j\sk\|_{l^2X^s}
\]
We can then write 
\begin{align*}
v_j\so-v_j\st=&\;\tilde S_j\left((e^{-\frac{1}{3}a_{<j-4}\so}-e^{-\frac{1}{3}a_{<j-4}\st})S_{<j-4}(e^{\frac{1}{3}a_{<j-4}\so})w_j\so)\right.\\
&\left.+\tilde S_j(e^{-\frac{1}{3}a_{<j-4}\st}S_{<j-4}(e^{\frac{1}{3}a_{<j-4}\so}-e^{\frac{1}{3}a_{<j-4}\st})w_j\so)\right.\\
&\left.+\tilde S_j(e^{-\frac{1}{3}a_{<j-4}\st}S_{<j-4}(e^{\frac{1}{3}a_{<j-4}\st})(w_j\so-w_j\st)\right)
\end{align*}
and have the estimate
\begin{align*}
&\|(v_j\so-v_j\st)-(w_j\so-w_j\st)\|_{l^2X^s}\\
&\lesssim C(\|a\so\|_{l^2X^\sigma},\|a\st\|_{l^2X^\sigma})(\|a\so\!-\!a\st\|_{l^2X^\sigma}\|v_j\so\|_{l^2X^s}\!+\!\|v_j\so\!-\!v_j\st\|_{l^2X^s})
\end{align*}
So we then have
\begin{align*}
&\|u_j\so-u_j\st\|_{l^2X^s}\\
&\lesssim C(\|a\so\|_{l^2X^\sigma}\!,\!\|a\st\|_{l^2X^\sigma})(\|a\so\!\!-\!a\st\|_{l^2X^\sigma}\|w_j\so\|_{l^2X^s}\!+\!\|w_j\so\!\!-\!w_j\st\|_{l^2X^s})\\
&\lesssim C(\|a\so\|_{l^2X^\sigma}\!,\!\|a\st\|_{l^2X^\sigma})(\|a\so\!-\!a\st\|_{l^2X^\sigma}\|v_j\so\|_{l^2X^s}\!+\!\|v_j\so\!-\!v_j\st\|_{l^2X^s})
\end{align*}
\end{proof}


\section{Rescaling}\label{sect:rescale}
As the \(l^2X^s,l^2Y^s\) spaces are adapted to the unit interval, we rescale the initial data to allow us to consider a small data problem on the unit time interval. Following Bejenaru-Tataru \cite{BT} we split the initial data into low and high frequency parts. As the large low frequency part is essentially stationary on the unit interval we freeze it a time \(t=0\) and solve for the high frequency component.

We rescale the initial data according to the nonlinearity
\[
u_0^{(k)}=2^{\lambda k}u_0(2^{-k}x)
\]
where, for
\[
F(u,u_x,u_{xx})=\sum\limits_{\alpha}c_\alpha u^{\alpha_0}u_x^{\alpha_1}u_{xx}^{\alpha_2}
\]
we define
\begin{equation}
\lambda=\max\left\{\frac{\beta_1+2\beta_2-3}{|\beta|-1}:|\beta|\geq2,\,\beta\leq\alpha,\,c_\alpha\neq0\right\}
\end{equation}
We then define the low and high frequency components of the rescaled initial data to be
\[
u_0^{(k)l}=S_0u_0^{(k)}\qquad u_0^{(k)h}=u_0^{(k)}-u_0^{(k)l}
\]

We have the following estimates for the low and high frequency components of the rescaled initial data.

\vspace{15pt}
\begin{lem}\label{lem:scaleests}~

a) (High frequency estimate)
\begin{equation}\label{est:rescaledHF}
\|u_0^{(k)h}\|_{l^2H^s}\lesssim2^{(\lambda+\frac{1}{2}-s)k}\|u_0\|_{l^2H^s}
\end{equation}

b) (Low frequency \(l^2H^s\) estimate) If \(s>\tfrac{1}{2}\) then for any \(\sigma\geq0\)
\begin{equation}\label{est:rescaledLF}
\|\partial_x^ru_0^{(k)l}\|_{l^2H^\sigma}\lesssim2^{(\lambda-\min\{s-\frac{1}{2},r\})k}\|u_0\|_{l^2H^s}
\end{equation}

\end{lem}

\begin{proof}~

a) This follows from the fact that
\[
\|u_0^{(k)h}\|_{l^2H^s}\sim\|u_0^{(k)h}\|_{H^s}
\]

b) We have
\begin{align*}
\|\partial_x^ru_0^{(k)l}\|_{l^2H^s}&\lesssim\|\partial_x^ru_0^{(k)l}\|_{l^2_0L^2}\\
&\lesssim\|\partial_x^ru_0^{(k)l}\|_{l^2_0L^\infty}\\
&\lesssim2^{(\lambda-r)k}\|\partial_x^rS_{\leq k}u_0\|_{l^2_{-k}L^\infty}\\
&\lesssim2^{(\lambda-r)k}\sum\limits_{j=0}^k2^{rj}\|S_ju_0\|_{l^2_{-k}L^\infty}\\
&\lesssim2^{(\lambda-r)k}\sum\limits_{j=0}^k2^{(r+\tfrac{1}{2})j}\|S_ju_0\|_{l^2_{-k}L^2}\\
&\lesssim 2^{(\lambda-\min\{s-\frac{1}{2},r\})k}\|u_0\|_{l^2H^s}
\end{align*}

\end{proof}
\vspace{15pt}

\begin{rk}\label{rk:lambda}
We note that for any of the possible nonlinearities \(F\), we have \(\lambda\in[-3,2)\). We also have that \(s_0\geq\lambda+\tfrac{1}{2}\), so provided \(s>s_0\) we can ensure \(\|u_0^{(k)h}\|_{l^2H^s}\) is arbitrarily small by choosing sufficiently large \(k\).
\end{rk}\vspace{15pt}

If \(u\) solves \eqref{eq:pde}, we rescale
\[
u^{(k)}(t,x)=2^{\lambda k}u(2^{-3k}t,2^{-k}x)
\]
and fix the low frequency component at time \(t=0\)
\[
v=u^{(k)}-u_0^{(k)l}
\]
Taking \(v_0=u_0^{(k)h}\), we then have that \(v\) solves the equation
\begin{equation}\label{eq:veqn}
\pde{(\partial_t+\partial_x^3)v=\tilde F(x,v,v_x,v_{xx})}{v(0)=v_0}
\end{equation}
where
\[
\tilde F=2^{(\lambda-3)k}F\left(v+u_0^{(k)l},2^k\partial_x(v+u_0^{(k)l}),2^{2k}\partial_x^2(v+u_0^{(k)l})\right)+\partial_x^3u_0^{(k)l}
\]

Due to the Mizohata condition \eqref{eq:Mz} we split the nonlinearity
\[
\tilde F(x,v,v_x,v_{xx})=B(x,v_x,v_{xx})+G(x,v,v_x,v_{xx})
\]
where
\begin{equation}\label{defn:badterms}
B(x,v_x,v_{xx})=c_12^{-\lambda k}(\partial_xu_0^{(k)l}v_{xx}+v_xv_{xx})+c_22^{(1-\lambda)k}(\partial_x^2u_0^{(k)l}v_{xx}+v_{xx}^2)
\end{equation}
contains the `bad' quadratic terms where two derivatives fall on one term. We have the following estimate for the `good' terms \(G\).

\vspace{15pt}\begin{propn}\label{propn:niceRHSterms}
For \(s>s_0\)
\begin{align}\label{est:Gest}
\|G&(x,v,v_x,v_{xx})\|_{l^2Y^s}\\&\lesssim C(\|u_0\|_{l^2H^s})\left(\|v\|_{l^2X^s}^2\langle\|v\|_{l^2X^s}\rangle^{m-2}+2^{-k}\|v\|_{l^2X^s}+2^{-\gamma k}\|u_0\|_{l^2H^s}\right)\notag
\end{align}
where \(m\) is the degree of \(F\) and
\begin{equation}\label{defn:gamma}
\gamma=\min(1,s-\lambda-\tfrac{1}{2})>0
\end{equation}

Further, if \(v\sk,G\sk\) correspond to initial data \(u_0\sk\) for \(i=1,2\), we have the following estimate for the difference
\begin{align}\label{est:Gdiff}
&\|G\so(x,v\so)-G\st(x,v\st)\|_{l^2Y^s}\\
&\lesssim C(\|u_0\sk\|_{l^2H^s}\!,\!\|v\sk\|_{l^2X^s})\left(\|v\so\!-\!v\st\|_{l^2X^s}\left(\|v\so\|_{l^2X^s}\!+\!\|v\st\|_{l^2X^s}\!+\!2^{-k}\right)\right.\notag\\
&\left.\quad+\|u_0\so-u_0\st\|_{l^2H^s}\left((\|v\so\|_{l^2X^s}+\|v\st\|_{l^2X^s})^2\notag\right.\right.\\
&\left.\left.\;\;\:\,+2^{-k}(\|v\so\|_{l^2X^s}+\|v\st\|_{l^2X^s})+2^{-\gamma k}\right)\right)\notag
\end{align}
\end{propn}

\begin{proof}
We note that by Remark \ref{rk:lambda} we have \(\gamma>0\). To estimate the inhomogeneous term \(\partial_x^3u_0^{(k)l}\) we use \eqref{est:rescaledLF}.
\[
\|\partial_x^3u_0^{(k)l}\|_{l^2Y^s}\lesssim2^{-(\min(s-\frac{1}{2},3)-\lambda)k}\|u_0\|_{l^2H^s}
\]
The remaining terms in \(G\) are of the form
\[
w_{\alpha\beta}=c_\alpha2^{(\lambda-\lambda|\alpha|+\alpha_1+2\alpha_2-3)k}(u_0^{(k)l})^{\alpha_0-\beta_0}(u_0^{(k)l})_x^{\alpha_1-\beta_1}(u_0^{(k)l})_{xx}^{\alpha_2-\beta_2}v^{\beta_0}v_x^{\beta_1}v_{xx}^{\beta_2}
\]
where \(0\leq\beta\leq\alpha\). We can use the algebra estimate \eqref{est:alg}, bilinear estimate \eqref{est:bil}, trilinear estimate \eqref{est:tri} and the low frequency estimates \eqref{est:LFXH} and \eqref{est:rescaledLF} to get
\[
\|w_{\alpha\beta}\|_{l^2Y^s}\lesssim2^{-\mu_{\alpha\beta}k}\|u_0\|_{l^2H^s}^{|\alpha|-|\beta|}\|v\|_{l^2X^s}^{|\beta|}
\]
where
\begin{align*}
\mu_{\alpha\beta}=&\,\lambda(|\beta|-1)-(\beta_1+2\beta_2-3)+(s-\tfrac{1}{2})(\alpha_0-\beta_0)\\
&+\min(s-\tfrac{3}{2},0)(\alpha_1-\beta_1)+\min(s-\tfrac{5}{2},0)(\alpha_2-\beta_2)\notag
\end{align*}
If \(|\beta|\geq2\), from the definition of \(\lambda\) we have,
\[
\mu_{\alpha\beta}\geq\lambda(1-|\beta|)-(\beta_1+2\beta_2-3)\geq0
\]
If \(|\beta|=1\) we have
\[
\mu_{\alpha\beta}\geq3-\beta_1-2\beta_2\geq1
\]
If \(|\beta|=0\) we have
\[
\mu_{\alpha\beta}\geq3-\lambda\geq1
\]

For \eqref{est:Gdiff} we write the difference \(G\so(x,v\so)-G\st(x,v\st)\) as a polynomial in \((u_0\sk)^{(k)l}\), \(v\sk\),\((u_0\so)^{(k)l}-(u_0\st)^{(k)l}\), \(v\so-v\st\) and use the same estimates as for \eqref{est:Gest}. (See \cite{HG} Lemma 5.2).
\end{proof}
\vspace{15pt}

\begin{rk}In the case that \(B\equiv0\) we can now apply a contraction mapping argument using Proposition \ref{propn:linLocWP} to prove Theorem \ref{thrm:main} (see \cite{HG} for example).\end{rk}


\section{Proof of Theorem \ref{thrm:main}}\label{sect:pf}

\subsection{The paradifferential decomposition}
We now consider the case that \(B\not\equiv0\). The difficulty here is that we cannot apply the bilinear estimates of Proposition \ref{propn:bil} when two derivatives fall at high frequency. However, by using a paradifferential decomposition we can consider an equation of the form \eqref{eq:paralinearpde} and use Proposition \ref{propn:modairyparaEST}.
 
We decompose \(B\) at frequency \(2^j\) as
\[
S_jB(x,v_x,v_{xx})=\partial_xa_{<j-4}S_jv_{xx}+[S_j,\partial_xa_{<j-4}]v_{xx}+b_j
\]
where
\[
a_{<j-4}(x,v)=c_12^{-\lambda k}\left(u_0^{(k)l}+S_{<j-4}v\right)+c_22^{(1-\lambda)k}\partial_x\left(u_0^{(k)l}+2S_{<j-4}v\right)
\]
and
\[
b_j(v)=c_12^{-\lambda k}S_j(S_{\geq j-4}(v_x)v_{xx})+c_22^{(1-\lambda)k}S_j((S_{\geq j-4}v_{xx})^2)
\]

Let
\[
H_j(x,v)=[S_j,\partial_xa_{<j-4}]v_{xx}+b_j(x,v)+S_jG(x,v,v_x,v_{xx})
\]
and \(H(x,v)=\sum H_j(x,v)\). The equation \eqref{eq:veqn} for \(v\) can then be written as
\begin{equation}\label{eq:paraveqn}
\pde{(\partial_t+\partial_x^3-T_{\partial_x a}\partial_x^2)v=H(x,v)}{v(0)=v_0}
\end{equation}
We have the following estimate for \(H\).

\vspace{15pt}\begin{propn}\label{propn:Hest}
For \(s>s_0\) and \(m\), \(\gamma\) as in Proposition \ref{propn:niceRHSterms}
\begin{align}\label{est:Hbnd}
\|H&(x,v)\|_{l^2Y^s}\\&\lesssim C(\|u_0\|_{l^2X^s})\left(\|v\|_{l^2X^s}^2\langle\|v\|_{l^2X^s}\rangle^{m-2}+2^{-k}\|v\|_{l^2X^s}+2^{-\gamma k}\|u_0\|_{l^2H^s}\right)\notag
\end{align}

Further, if \(v\sk,H\sk\) correspond to initial data \(u_0\sk\) for \(i=1,2\), we have the following estimate for the difference
\begin{align}\label{est:Hdiff}
&\|H\so(x,v\so)-H\st(x,v\st)\|_{l^2Y^s}\\
&\lesssim C(\|u_0\sk\|_{l^2H^s}\!,\!\|v\sk\|_{l^2X^s})\left(\|v\so\!-\!v\st\|_{l^2X^s}(\|v\so\|_{l^2X^s}\!+\!\|v\st\|_{l^2X^s}\!+\!2^{-k})\right.\notag\\
&\left.\quad+\|u_0\so-u_0\st\|_{l^2H^s}\left((\|v\so\|_{l^2X^s}+\|v\st\|_{l^2X^s})^2\notag\right.\right.\\
&\left.\left.\;\;\:\,+2^{-k}(\|v\so\|_{l^2X^s}+\|v\st\|_{l^2X^s})+2^{-\gamma k}\right)\right)\notag
\end{align}
\end{propn}
\begin{proof}
Define
\begin{equation}\label{defn:sigma}
\sigma=\left\{\begin{array}{ll}s&\quad\textrm{if }c_2=0\vspace{5pt}\\s-1&\quad\textrm{if }c_2\neq0\end{array}\right.
\end{equation}
Then, by the commutator estimate \eqref{est:com} and the low frequency estimate \eqref{est:rescaledLF} we have
\begin{align*}
\|[S_j,\partial_xa_{<j-4}]v_{xx}\|_{l^2Y^s}&\lesssim\|\partial_x a\|_{l^2X^{\sigma-1}}\|\tilde S_jv\|_{l^2X^s}\\
&\lesssim(2^{-k}\|u_0\|_{l^2H^s}+\|v\|_{l^2X^s})\|\tilde S_jv\|_{l^2X^s}
\end{align*}
From the frequency localised bilinear estimates \eqref{est:bilLH} and \eqref{est:bilHH} we have
\[
\|\sum\limits_{j\geq0}b_j(x,v)\|_{l^2Y^s}\lesssim\|v\|_{l^2X^s}^2
\]
where we have used that if \(c_1\neq0\) then \(\lambda\geq0\) and if \(c_2\neq0\) then \(\lambda\geq1\). The remaining terms in \eqref{est:Hbnd} can then be estimated using Proposition \ref{propn:niceRHSterms}.

For the estimate \eqref{est:Hdiff}, as in Proposition \ref{propn:niceRHSterms}, we write the difference \(H\so(x,v\so)-H\st(x,v\st)\) in terms of \((u_0\sk)^{(k)l}\), \(v\sk\),\((u_0\so)^{(k)l}-(u_0\st)^{(k)l}\), \(v\so-v\st\) and use the same estimates as for \eqref{est:Hbnd}.
\end{proof}
\vspace{15pt}

\subsection{The solution map}
Let \(\sigma\) be as in \eqref{defn:sigma}, \(\gamma\) as in \eqref{defn:gamma} and \(v_0=u_0^{(k)h}\). Define
\[
\xk=\{v\in l^2X^s:\|v\|_{l^2X^s}\leq2^{-\frac{1}{2}\gamma k}\|u_0\|_{l^2H^s}\}
\]
We note that if \(v\in\xk\), from the low frequency estimate \eqref{est:rescaledLF}
\[
\|a(v)\|_{l^2X^\sigma}\lesssim \|u_0\|_{l^2H^s}+\|v\|_{l^2X^s}\lesssim\|u_0\|_{l^2X^s}
\]
and
\[
\|\partial_xa(v)\|_{l^2X^{\sigma-1}}\lesssim 2^{-k}\|u_0\|_{l^2H^s}+\|v\|_{l^2X^s}\lesssim 2^{-\frac{1}{2}\gamma k}\|u_0\|_{l^2H^s}
\]
In particular, if \(\delta=\delta(s,\sigma,\|u_0\|_{l^2H^s})\) is as in Proposition \ref{propn:modairyparaEST}, for sufficiently large \(k\),
\[
\|\partial_xa(v)\|_{l^2X^{\sigma-1}}\leq\delta
\]
Suppose \(v\in\xk\) also satisfies
\begin{equation}\label{cond:paracondi}
\|T_{(\partial_t+\partial_x^3)a(v)}\|_{l^2X^s\rightarrow l^2Y^s}\leq\delta
\end{equation}
then by Proposition \ref{propn:modairyparaEST} we can find a solution \(w=\mathcal{T}(v)\) to the equation
\[
\pde{(\partial_t+\partial_x^3-T_{\partial_xa(v)}\partial_x^2)w=H(x,v)}{w(0)=v_0}
\]
satisfying
\begin{equation}\label{est:tvbound}
\|w\|_{l^2X^s}\lesssim C(\|u_0\|_{l^2X^s})(\|v_0\|_{l^2H^s}+\|H(x,v)\|_{l^2Y^s})
\end{equation}

\vspace{15pt}\begin{propn}
Suppose \(s>s_0\), \(v\in\xk\) satisfies \eqref{cond:paracondi} and \(w=\mathcal{T}(v)\), then for sufficiently large \(k\),

\begin{equation}\label{est:whyp1}
\|w\|_{l^2X^s}\leq 2^{-\frac{1}{2}\gamma k}\|u_0\|_{l^2H^s}
\end{equation}
and
\begin{equation}\label{est:smallparanorm}
\|T_{(\partial_t+\partial_x^2)a(w)}\|_{l^2X^s\rightarrow l^2Y^s}\leq2^{-\gamma k}C(\|u_0\|_{l^2H^s})
\end{equation}
So, for sufficiently large \(k\), \(w\in\xk\) and satisfies \eqref{cond:paracondi}.

Further, if \(v\sk\in\xk\sk\) satisfy \eqref{cond:paracondi} for \(\delta\sk\geq0\) and \(w\sk=\mathcal{T}\sk(v\sk)\) where \(\xk\sk,\mathcal{T}\sk\) correspond to initial data \(u_0\sk\) for \(i=1,2\), we have the following estimate for the difference
\begin{align}\label{est:paranormdiff}
\|T&_{(\partial_t+\partial_x^3)(a\so(w\so)-a\st(w\st))}\|_{l^2X^s\rightarrow l^2Y^s}\\
&\qquad\qquad\lesssim C(\|u_0\so\|_{l^2H^s},\|u_0\st\|_{l^2H^s})\left(2^{-\gamma k}\|u_0\so-u_0\st\|_{l^2H^s}+\right.\notag\\
&\qquad\qquad\quad\left.+2^{-\frac{1}{2}\gamma k}\|v\so-v\st\|_{l^2X^s}+2^{-\frac{1}{2}\gamma k}\|w\so-w\st\|_{l^2X^s}\right)\notag
\end{align}
\end{propn}

\begin{proof}

From the high frequency estimate \eqref{est:rescaledHF} we have
\[
\|v_0\|_{l^2H^s}\lesssim 2^{-\gamma k}\|u_0\|_{l^2H^s}
\]
and from Proposition \ref{propn:Hest} and the hypothesis \(\|v\|_{l^2X^s}\lesssim 2^{-\frac{1}{2}\gamma k}\) we have
\[
\|H(x,v)\|_{l^2Y^s}\lesssim C(\|u_0\|_{l^2H^s})2^{-\gamma k}\|u_0\|_{l^2H^s}
\]
The estimate \eqref{est:whyp1} then follows from \eqref{est:tvbound} for sufficiently large \(k\).

To prove \eqref{est:smallparanorm} we take \(z\in l^2X^s\) and consider
\begin{align*}
\|(\partial_t&+\partial_x^3)(a_{<j-4}(w))S_jz\|_{l^2Y^s}\\
\lesssim&\;\|(c_12^{-\lambda k}+c_22^{(1-\lambda)k}\partial_x)(\partial_x^3u_0^{(k)l})S_jz\|_{l^2Y^s}\\
&\;+\|(c_12^{-\lambda k}+2c_22^{(1-\lambda)k}\partial_x)(S_{<j-4}H(x,v))S_jz\|_{l^2Y^s}\\
&\;+\|(c_12^{-\lambda k}+2c_22^{(1-\lambda)k}\partial_x)(S_{<j-4}(T_{\partial_xa(v)}\partial_x^2w))S_jz\|_{l^2Y^s}
\end{align*}
Using the frequency localised bilinear estimate \eqref{est:bilLH} we have
\begin{align*}
&\|(c_12^{-\lambda k}+c_22^{(1-\lambda)k}\partial_x)(\partial_x^3u_0^{(k)l})S_jz\|_{l^2Y^s}\\
&+\|(c_12^{-\lambda k}+2c_22^{(1-\lambda)k}\partial_x)(S_{<j-4}H(x,v))S_jz\|_{l^2Y^s}\\
&\;\lesssim(\|\partial_x^3u_0^{(k)l}\|_{l^2X^{s-2}}+\|S_{<j-4}H(x,v)\|_{l^2X^{s-2}})\|S_jz\|_{l^2X^s}
\end{align*}
From the low frequency estimate \eqref{est:rescaledLF} we have
\[
\|\partial_x^3u_0^{(k)l}\|_{l^2X^{s-2}}\lesssim2^{(\lambda-3)k}\|u_0\|_{l^2H^s}\lesssim2^{-k}\|u_0\|_{l^2H^s}
\]
Replacing the frequency localised bilinear estimates \eqref{est:bilLH} and \eqref{est:bilHH} and the commutator estimate \eqref{est:com} in Proposition \ref{propn:Hest} by the frequency localised bilinear estimates \eqref{est:LHbilX} and \eqref{est:HHbilX} we have
\[
\|H(x,v)\|_{l^2X^{s-2}}\lesssim(2^{-k}\|u_0\|_{l^2H^s}+\|v\|_{l^2X^s})\|v\|_{l^2X^s}+\|G(x,v)\|_{l^2X^{s-2}}
\]
Similarly, replacing the bilinear estimate \eqref{est:bil} and the trilinear estimate \eqref{est:tri} in Proposition \ref{propn:niceRHSterms} by the algebra estimate \eqref{est:alg} we have
\begin{align*}
\|G&(x,v)\|_{l^2X^{s-2}}\\&\lesssim C(\|u_0\|_{l^2H^s})\left(\|v\|_{l^2X^s}^2\langle\|v\|_{l^2X^s}\rangle^{m-2}+2^{-k}\|v\|_{l^2X^s}+2^{-\gamma k}\|u_0\|_{l^2H^s}\right)
\end{align*}
So combining these we have
\[
\|H(x,v)\|_{l^2X^{s-2}}\lesssim2^{-\gamma k}C(\|u_0\|_{l^2H^s})
\]
For the remaining term we use the frequency localised trilinear estimate \eqref{est:triLHH} with the estimate \eqref{est:whyp1} to get,
\begin{align*}
\|(c_12^{-\lambda k}+2c_22^{(1-\lambda)k}\partial_x)&(S_{<j-4}(T_{\partial_xa(v)}\partial_x^2w))S_jz\|_{l^2Y^s}\\
&\lesssim\|\partial_xa(v)\|_{l^2X^{\sigma-1}}\|w\|_{l^2X^s}\|S_jz\|_{l^2X^s}\\
&\lesssim2^{-\gamma k}\|u_0\|^2_{l^2H^s}\|S_jz\|_{l^2X^s}
\end{align*}

To prove \eqref{est:paranormdiff} we consider the difference
\begin{align*}
(\partial_t+\partial_x^3)&(a\so(v\so)-a\st(v\st))\\
&\qquad=(c_12^{-\lambda k}+c_22^{(1-\lambda)k}\partial_x)\left(\partial_x^3(u_0\so)^{(k)l}-\partial_x^3(u_0\st)^{(k)l}\right)\\
&\qquad\quad+(c_12^{-\lambda k}+2c_22^{(1-\lambda)k}\partial_x)\left(T_{\partial_xa\so(v\so)}\partial_x^2w\so\right.\\
&\qquad\quad\left.-T_{\partial_xa\st(v\st)}\partial_x^2w\st+H\so(x,v\so)-H\st(x,v\st)\right)
\end{align*}
As above we have
\[
\|\partial_x^3(u_0\so)^{(k)l}-\partial_x^3(u_0\st)^{(k)l}\|_{l^2X^{s-2}}\lesssim2^{-k}\|u_0\so-u_0\st\|_{l^2H^s}
\]
and as in Propositions \ref{propn:niceRHSterms}, \ref{propn:Hest} we can write the difference
\[
H\so(x,v\so)-H\st(x,v\st)
\]
as a polynomial in \((u_0\sk)^{(k)l}\), \(v\sk\), \((u_0\so)^{(k)l}-(u_0\st)^{(k)l}\), and \(v\so-v\st\) and apply the same estimates as for \eqref{est:smallparanorm} to get
\begin{align*}
&\|H\so(x,v\so)-H\st(x,v\st)\|_{l^2X^{s-2}}\\
&\lesssim C(\|u_0\sk\|_{l^2H^s}\!,\!\|v\sk\|_{l^2X^s})\left(\|v\so\!-\!v\st\|_{l^2X^s}(\|v\so\|_{l^2X^s}\!+\!\|v\st\|_{l^2X^s}\!+\!2^{-k})\right.\\
&\quad\left.+\|u_0\so-u_0\st\|_{l^2H^s}\left((\|v\so\|_{l^2X^s}+\|v\st\|_{l^2X^s})^2\notag\right.\right.\\
&\quad\left.\left.+2^{-k}(\|v\so\|_{l^2X^s}+\|v\st\|_{l^2X^s})+2^{-\gamma k}\right)\right)\\
&\lesssim C(\|u_0\sk\|_{l^2H^s})(2^{-\frac{1}{2}\gamma k}\|v\so-v\st\|_{l^2X^s}+2^{-\gamma k}\|u_0\so-u_0\st\|_{l^2H^s})
\end{align*}
For the remaining term we write
\begin{align*}
T_{\partial_xa\so(v\so)}\partial_x^2w\so-T_{\partial_xa\st(v\st)}\partial_x^2w\st&=T_{\partial_xa\so(v\so)-\partial_xa\st(v\st)}\partial_x^2w\so\\
&\quad+T_{\partial_xa\st(v\st)}\partial_x^2(w\so-w\st)
\end{align*}
Applying the frequency localised trilinear estimate \eqref{est:triLHH} as above we then have
\begin{align*}
\|S_{<j-4}&(T_{\partial_xa\so(v\so)}\partial_x^2w\so-T_{\partial_xa\st(v\st)}\partial_x^2w\st)S_jz\|_{l^2Y^s}\\
&\lesssim \|\partial_xa\so(v\so)-\partial_xa\st(v\st)\|_{l^2X^{\sigma-1}}\|w\so\|_{l^2X^s}\|z_j\|_{l^2X^s}\\
&\quad+\|\partial_xa\st(v\st)\|_{l^2X^{\sigma-1}}\|w\so-w\st\|_{l^2X^s}\|z_j\|_{l^2X^s}\\
&\lesssim(2^{-k}\|u_0\so\!-\!u_0\st\|_{l^2H^s}+\|v\so\!-\!v\st\|_{l^2X^s})2^{-\frac{1}{2}\gamma k}\|u_0\st\|_{l^2H^s}\|z_j\|_{l^2X^s}\\
&\quad+2^{-\frac{1}{2}\gamma k}\|u_0\st\|_{l^2H^s}\|w\so-w\st\|_{l^2X^s}\|z_j\|_{l^2X^s}
\end{align*}

\end{proof}
\vspace{15pt}

\subsection{Existence of a solution}
Let \(v\sm\equiv0\) and let
\[
v^{\snp}=\mathcal{T}(v\sn)
\]
Then \(v\sn\in\xk\) and satisfies \eqref{cond:paracondi} for \(n\geq0\). To estimate the difference \(v\snp-v\sn\) we use Proposition \ref{propn:modAiryparadiffs} to get
\begin{align*}
&\|v\snp-v\sn\|_{l^2X^s}\\
&\lesssim C(\|u_0\|_{l^2H^s})\left(\|H(v\sn)-H(v\snm)\|_{l^2Y^s}+\left(\|a(v\sn)-a(v\snm)\|_{l^2X^\sigma}\right.\right.\\
&\quad\left.\left.+\|T_{(\partial_t+\partial_x^3)(a(v\sn)-a(v\snm))}\|_{l^2X^s\rightarrow l^2Y^s}\right)\|H(v\sn)\|_{l^2Y^s}\!\right)
\end{align*}

From Proposition \ref{propn:Hest} we have
\begin{align*}
\|&H(v\sn)-H(v\snm)\|_{l^2Y^s}\\&\lesssim C(\|u_0\|_{l^2H^s},\|v\sn\|_{l^2X^s},\|v\snm\|_{l^2X^s})\|v\sn-v\snm\|_{l^2X^s}\left(\|v\sn\|_{l^2X^s}\right.\\
&\quad\left.+\|v\snm\|_{l^2X^s}+2^{-k}\right)\\
&\lesssim C(\|u_0\|_{l^2H^s})2^{-\frac{1}{2}\gamma k}\|v\sn-v\snm\|_{l^2X^s}
\end{align*}
We also have
\[
\|a(v\sn)-a(v\snm)\|_{l^2X^\sigma}\lesssim\|v\sn-v\snm\|_{l^2X^s}
\]

From \eqref{est:paranormdiff} we have
\begin{align*}
\|&T_{(\partial_t+\partial_x^3)(a(v\sn)-a(v\snm))}\|_{l^2X^\sigma\rightarrow l^2Y^\sigma}\\
&\lesssim C(\|u_0\|_{l^2H^s})(2^{-\frac{1}{2}\gamma k}\|v\snm-v\snmm\|_{l^2X^s}+2^{-\frac{1}{2}\gamma k}\|v\sn-v\snm\|_{l^2X^s})
\end{align*}
So
\begin{align*}
\|&v\snp-v\sn\|_{l^2X^s}\\
&\lesssim C(\|u_0\|_{l^2H^s})(2^{-\frac{1}{2}\gamma k}\|v\sn-v\snm\|_{l^2X^s}+2^{-\frac{1}{2}\gamma k}\|v\snm-v\snmm\|_{l^2X^s})
\end{align*}
So provided \(k\) is sufficiently large, the sequence converges to a solution \linebreak \(v\in\xk\) to \eqref{eq:veqn}. Adding the low frequency component of the initial data \(u_0^{(k)l}\) and rescaling we get a solution \(u\in C([0,2^{-3k}],H^s)\) to \eqref{eq:pde}.

\subsection{Uniqueness and Lipschitz dependence on initial data}

Suppose we have solutions \(u\sk\) for \(i=1,2\) to \eqref{eq:pde} corresponding to initial data \(u_0\sk\). After rescaling and subtracting the low frequency component we have \(v\sk\) satisfying
\begin{equation}\label{est:rszbd}
\|v\sk\|_{l^2X^s}\leq 2^{-\tfrac{1}{2}\gamma k}\|u_0\sk\|_{l^2H^s}
\end{equation}
and
\begin{equation}
\pde{(\partial_t+\partial_x^3-T_{\partial_xa\sk(v\sk)}\partial_x^2)v\sk=H\sk(v\sk)}{v\sk(0)=v_0\sk}
\end{equation}

Applying Proposition \ref{propn:modAiryparadiffs} we have the estimate
\begin{align*}
&\|v\so-v\st\|_{l^2X^s}\\
&\lesssim\!\!C(\|u_0\so\|_{l^2X^s}\!,\!\|u_0\st\|_{l^2X^s})\!\left(\!\|v_0\so\!\!-\!v_0\st\|_{l^2H^s}\!+\!\|H\so(v\so)\!-\!H\st(v\st)\|_{l^2Y^s}\right.\\
&\left.\;\;+(\|a\so(v\so)-a\st(v\st)\|_{l^2X^\sigma}\right.\\
&\left.\;\;+\|T_{(\partial_t+\partial_x^3)(a\so(v\so)-a\st(v\st))}\|_{l^2X^s\rightarrow l^2Y^s})(\|v_0\so\|_{l^2H^s}\!+\!\|H\so(v\so)\|_{l^2Y^s})\!\right)
\end{align*}
From the estimate \eqref{est:Hdiff} we have
\begin{align*}
&\|H\so(v\so)-H\st(v\st)\|_{l^2Y^s}\\
&\lesssim C(\|u_0\so\|_{l^2H^s},\|u_0\st\|_{l^2H^s})\left(\|v\so\!-\!v\st\|_{l^2X^s}(\|v\so\|_{l^2X^s}\!+\!\|v\st\|_{l^2X^s}\!+\!2^{-k})\right.\\
&\left.\quad+\|u_0\so-u_0\st\|_{l^2H^s}\left((\|v\so\|_{l^2X^s}+\|v\st\|_{l^2X^s})^2\notag\right.\right.\\
&\left.\left.\;\;\:\,+2^{-k}(\|v\so\|_{l^2X^s}+\|v\st\|_{l^2X^s})+2^{-\gamma k}\right)\right)\notag\\
&\lesssim C(\|u_0\so\|_{l^2H^s},\|u_0\st\|_{l^2H^s})\left(2^{-\gamma k}\|u_0\so-u_0\st\|_{l^2X^s}+2^{-\frac{1}{2}\gamma k}\|v\so-v\st\|_{l^2X^s}\right)
\end{align*}
We have
\begin{align*}
\|&a\so(v\so)-a\st(v\st)\|_{l^2X^\sigma}\\
&\lesssim C(\|u_0\so\|_{l^2H^s},\|u_0\st\|_{l^2H^s})(\|u_0\so-u_0\st\|_{l^2H^s}+\|v\so-v\st\|_{l^2X^s})
\end{align*}
and from the estimate \eqref{est:paranormdiff},
\begin{align*}
&\|T_{(\partial_t+\partial_x^3)(a\so(v\so)-a\st(v\st))}\|_{l^2X^s\rightarrow l^2Y^s}\\
&\lesssim C(\|u_0\so\|_{l^2H^s},\|u_0\st\|_{l^2H^s})\left(2^{-\gamma k}\|u_0\so-u_0\st\|_{l^2X^s}+2^{-\frac{1}{2}\gamma k}\|v\so-v\st\|_{l^2X^s}\right)
\end{align*}

We then have the estimate
\begin{align*}
\|&v\so-v\st\|_{l^2X^s}\\
&\lesssim C(\|u_0\so\|_{l^2H^s},\|u_0\st\|_{l^2H^s})(2^{-\gamma k}\|u_0\so-u_0\st\|_{l^2H^s}+2^{-\frac{1}{2}\gamma k}\|v\so-v\st\|_{l^2X^s})
\end{align*}
and so sufficiently large \(k\) we have
\begin{align*}
\|u\so-u\st\|_{l^2X^s}&\lesssim\|u_0\so-u_0\st\|_{l^2H^s}+\|v\so-v\st\|_{l^2X^s}\\
&\lesssim\|u_0\so-u_0\st\|_{l^2X^s}
\end{align*}
so the solution map is locally Lipschitz.

\section*{Acknowledgements}The author would like to thank his advisor, Daniel Tataru, for suggesting the problem, his guidance and several suggestions for the proof.

\bibliographystyle{amsplain}

\end{document}